\documentclass[a4paper,10pt]{article}

\usepackage{amsthm}
\usepackage{amsmath}
\usepackage{amssymb}
\usepackage{amsfonts}
\usepackage{fancybox}

\theoremstyle{definition}

\newtheorem{df}{Definition}[section]
\newtheorem{thm}[df]{Theorem}
\newtheorem{lem}[df]{Lemma}
\newtheorem{cor}[df]{Corollary}
\newtheorem{prop}[df]{Proposition}

\newtheorem{rem}[df]{Remark}
\newtheorem{eg}[df]{Example}

\newtheorem{nt}[df]{Notation}
\newtheorem{claim}{Claim}
\newcommand\resetclaim{\setcounter{claim}{0}}

\makeatletter
  
  \@addtoreset{equation}{section}
\makeatother

\newcommand\mc{\mathclose}

\newcommand\ph{\varphi}
\newcommand\rest{\upharpoonright}
\newcommand\ra{\rightarrow}

\newcommand\iso{{\simeq}}
\newcommand\id{{\rm id}}
\newcommand\sm{\setminus}

\newcommand\N{\mathbb{N}}

\newcommand\Con{{{\rm Con}}}

\newcommand\Clo{{\rm Clo}}

\newcommand\Inv{{\rm Inv}}
\newcommand\Pol{{\rm Pol}}
\newcommand\J{{\rm J}}
\newcommand\M{{\rm M}}

\newcommand\eps{\varepsilon}

\newcommand\Ess{{\rm Ess}}
\newcommand\Ir{{\rm Ir}}

\title{A ``classification" of congruence primal arithmetical algebras}
\author{Shohei Izawa
}
\date{}

\begin{document}
\maketitle
\begin{abstract}
We classify essential algebras whose irredundant non-refinable covers consist of primal algebras.
The proof is obtained by constructing one to one
correspondence between such algebras and partial orders on finite sets.
Further, we prove that for a finite algebra, it has an irredundant non-refinable cover consists of primal algebras if and only if
it is the both congruence primal and arithmetical.
Finally, we obtain combinatorial description of congruence primal arithmetical algebras.
\end{abstract}

\section{Introduction.}
\label{s-introduction}

Relational structure theory is a branch of universal algebra, particularly studying structure of finite algebras.
In relational structure theory, we consider decomposition of algebras into a family of smaller algebras, which is called an irredundant non-refinable cover.
We may aim to classify finite algebras by considering the converse process, composition of ``irreducible algebras".
As pointed out in \cite{Iza}, the composition process is divided into two parts:
\begin{enumerate}
\item
For a given family ${\cal U}$ of irreducible algebras, classifying matrix product $A$ of ${\cal U}$. (Such algebras are said to be essential.)
\item
For each essential algebra $A$, classifying algebras categorically equivalent to $A$.
\end{enumerate}

Categorically equivalent algebras have very many common properties and
structures, and hence, obtaining classification of algebras compatible with categorical structure is an important task.

In this article, we consider the case that the family ${\cal U}$ of irreducible algebras consists of several two-element primal algebras.
Primality is a condition that means categorical structure is ``the simplest".
In this sense, the topic of this paper may be said a study of
the easiest case of application of relational structure theory to
classification of algebras.

In Section \ref{s-matrix-product-primal}, we consider Step (1) and
in Section \ref{s-cong-primal-arith}, we consider Step (2).
In the both steps, we obtain some ``classification" of objective algebras.
Concretely, in Section \ref{s-matrix-product-primal}, we show essential algebras constructed from primal algebras are bijectively corresponded to finite posets.
In Section \ref{s-cong-primal-arith}, we show ``inductive description" of finite algebras that have irredundant non-refinable covers that consist of primal algebras.

We also show in Section \ref{s-cong-primal-arith} that a finite algebra $A$ has an irredundant non-refinable cover that consists of primal algebras if and only if the algebra $A$ satisfies two algebraic conditions; congruence primality and arithmeticity.
Thus the main result of Section \ref{s-cong-primal-arith} also gives a classification of congruence primal arithmetical algebras.

\section{Preliminaries.}\label{s-preliminaries}
In this section, we quickly review basic definitions used in the following sections.
In this paper, we write $A\subset B$ the statement
$\forall x; x\in A\Rightarrow x\in B$.
The statement $A\subset B$ and $A\neq B$ is written by $A\subsetneq B$.
We write $\N$ the set of non-negative integers.
We write $|X|$ the cardinality of a set $X$.
We write ${\cal P}(X)$ the power set of a set $X$, namely, ${\cal P}(X):=\{Y\mid Y\subset X\}$.
We also use the following notations.

\begin{nt}
For sets $A$ and $I$, we write $A^I$ the set of all maps $I\ra A$.
The case $I=\{1,\dots,n\}$, we simply write $A^n$ instead of $A^I$.

We write ${\cal O}_A^{(n)}$ the set $A^{A^n}$ of $n$-ary operations on $A$
and ${\cal O}_A:=\bigcup_{n\in\N}{\cal O}_A^{(n)}$.
${\cal O}^{(n)}_{\{0,\dots,m-1\}}$ and ${\cal O}_{\{0,\dots,m-1\}}$ are simply written by ${\cal O}^{(n)}_m$ and ${\cal O}_m$ respectively.
\end{nt}

\begin{nt}
Let $A$ be a set. We write $\Delta^A$ the \emph{diagonal relation}
$\{(a,a)\in A^2\mid a\in A\}$.
For a set $I$ and an equivalence relation ${\cal E}$
on $I$, we define
\[
\Delta_{{\cal E}}^A:=\{(a_i)_{i\in I}\in A^I\mid (i,j)\in {\cal E}\Rightarrow a_i=a_j\}.
\]
The case the underlying set $A$ is clear from the context, we may omit the superscript $A$.
\end{nt}

\begin{nt}
Let $n,m$ be non-negative integers, $A_1,\ldots,A_n$ be sets and $r_i\subset A_i^m$ for $i\in \{1,\ldots,n\}$.
Then we write $\prod_{i=1}^n r_i$
the $m$-ary relation on the product set $\prod_{i=1}^n A_i$ that consists of all elements 
\[
\begin{pmatrix}
\begin{pmatrix} x_{11}\\ \vdots \\ x_{n1}\end{pmatrix}
,\ldots,
\begin{pmatrix} x_{1m}\\ \vdots \\ x_{nm}\end{pmatrix}
\end{pmatrix}
\in \left(\prod_{i=1}^n A_i\right)^m 
\]
that satisfy $(x_{i1},\ldots,x_{im})\in r_i$ for all $i\in\{1,\ldots,n\}$.
\end{nt}

\subsection{Preliminaries from clone theory}
In this subsection, we fix notions of clones, algebras and related concepts.

\begin{df}
Let $A$ be a set.
A set $C$ of operations on $A$,
namely $C\subset {\cal O}_A$, is said to be
a(n operational) \emph{clone} on $A$ 
if the following conditions hold:
\begin{enumerate}
\item
For each $n\in\N$ and $i\in\{1,\dots,n\}$, the $i$-th projection $(a_1,\dots,a_n)\mapsto a_i$ belongs to $C$.
\item
For each $n,m\in\N$, $f\in C\cap{\cal O}_A^{(n)}$ and
$g_i\in C\cap {\cal O}_A^{(m)}$ for $i\in\{1,\dots,n\}$, the composition
\[
f\circ (g_i)_{i=1}^n:(a_j)_{j=1}^m\mapsto 
f(g_i(a_j)_{j=1}^m)_{i=1}^n
\]
belongs to $C$.
\end{enumerate}
\end{df}

\begin{df}
Let $A$ be a set.
A set $R$ of relations on $A$,
that is, $R\subset \bigcup_{n\in\N}{\cal P}(A^{n})$, is said to be
a \emph{relational clone} on $A$
if the following conditions hold:
\begin{itemize}
\item
$\Delta^A\in R$.
\item
If $\{r_k\}_{k=1}^K\subset R$, $r_k\subset A^{n_k}$ and
a relation $r\subset A^{m}$ is described as in the form
\[
r=
\left\{(a_{j})_{j=1}^m\in A^{m}\left| 
  \exists (a_{j})_{j=m+1}^N 
  \bigwedge_{k=1}^K (a_{f_k(j)})_{j=1}^{n_k}\in r_{k}\right.\right\},
\]
where and $f_k:\{1,\dots,n_k\}\ra \{1,\dots,N\}$ and $N\in\N$, then $r\in R$ holds. (This condition is referred as ``$r$ is defined from $\{r_k\}_{k=1}^K$
by primitive positive definition.")
\end{itemize}
\end{df}

\begin{df}
Let $A$ be a set and $n,m\in\N$.
\begin{enumerate}
\item
Let $F$ be a set of operations on $A$.
An $m$-ary relation $r$ on $A$ is said to be \emph{invariant} to $F$ if
\[
\forall i\in\{1,\dots,n\}; (a_{ij})_{j=1}^m\in r
\ \  \Longrightarrow\ \ 
(f(a_{ij})_{i=1}^n)_{j=1}^m\in r
\]
hold for all $f\in F$, where $n$ is the arity of $f$.
The set of all invariant relations of $F$ is denoted by $\Inv(F)$.
We define $\Inv_{m}(F):=\Inv(F)\cap {\cal P}(A^m)$.
\item
Let $R$ be a set of relations on $A$.
An $n$-ary operation $f$ on $A$ is said to be a \emph{polymorphism} of $R$ if
\[
\forall i\in\{1,\dots,n\}; (a_{ij})_{j=1}^m\in r
\ \  \Longrightarrow\ \ 
(f(a_{ij})_{i=1}^n)_{j=1}^m\in r
\]
hold for all $r\in R$, where $m$ is the arity of $r$.
The set of all polymorphisms of $R$ is denoted by $\Pol(R)$.
We define $\Pol_{n}(R):=\Pol(R)\cap {\cal O}_A^{(n)}$.
\end{enumerate}
\end{df}



It is known that there is a natural one to one correspondence between clones and relational clones on a fixed finite set. The correspondence is given by $\Pol$ and $\Inv$ defined above.
\begin{thm}[\cite{BKKR},\cite{Gei}]\label{pol-inv-finite}
Let $A$ be a finite set.
\begin{enumerate}
\item
For any set $F$ of operations on $A$, $\Inv(F)$ is a relational clone on $A$.
\item
For any set $H$ of relations on $A$, $\Pol(H)$ is a clone on $A$.
\item
For a clone $C$ on $A$, $\Pol(\Inv(C))=C$ holds.
\item
For a relational clone $R$ on $A$, $\Inv(\Pol(R))=R$ holds.
\end{enumerate}
\end{thm}

\begin{df}
A \emph{non-indexed algebra} (or simply an \emph{algebra}) is a pair of $(A,C)$ of a set $A$ and a clone on $A$.
$A$ and $C$ are called the underlying set of the algebra and the set of term operations of the algebra respectively.
Usually, an algebra is denoted by its underlying set $A$ and the clone of term operations of it is denoted by $\Clo(A)$.
We also use the notation $\Clo_n(A):=\Clo(A)\cap {\cal O}_A^{(n)}$.
\end{df}

\begin{df}
The algebra of the form $(A,{\cal O}_A)$, where $A$ is a finite set, is said to be \emph{primal}.
\end{df}
\begin{rem}
If $A$ is a finite set, the relational clone corresponding to ${\cal O}_A$ is described as
\[
\Inv_n({\cal O}_A)=\{\Delta_{{\cal E}}^A\mid {\cal E}\text{ is an equivalence relation on }\{1,\dots,n\}\}.
\]
\end{rem}

\begin{df}
Let $A$ and $B$ be algebras. An \emph{isomorphism}
between $A$ and $B$ as non-indexed algebras is 
a bijection 
$\ph\colon A\ra B$
that satisfies the following conditions:
\begin{itemize}
\item
For any $n\in \N$ and $f\in \Clo_n(A)$, 
\[
\ph\circ f\circ \overline{\ph^{-1}}:=[(b_1,\ldots,b_n)\mapsto \ph(f(\ph^{-1}(b_1),\ldots,\ph^{-1}(b_n)))]\in \Clo_n(B).
\]
\item
For any $n\in \N$ and $g\in \Clo_n(B)$, 
\[
\ph^{-1}\circ g\circ \overline{\ph}:=[(a_1,\ldots,a_n)\mapsto \ph^{-1}(f(\ph(a_1),\ldots,\ph(a_n)))]\in \Clo_n(A).
\]
\end{itemize}
The algebra $A$ is said to be \emph{isomorphic}
to $B$ as non-indexed algebras if there
exists an isomorphism between $A$ and $B$.
\end{df}

\subsection{Preliminaries from relational structure theory}

In this subsection, we describe basic definitions and facts of relational structure theory.

\begin{df}
[{\cite[Definition 2.4, 3.1]{Kea}}]
Let $A$ be an algebra.
\begin{enumerate}
\item
The set of \emph{idempotent term operations} of $A$ is written by ${\bf E}(A)$,
i.e., ${\bf E}(A):=\{e\in\Clo_1(A)\mid A\models e^2(x)=e(x)\}$.
\item
$U\subset A$ is said to be a \emph{neighbourhood} of $A$ if there is $e\in {\bf E}(A)$ such that $e(A)=U$. 
The set of all neighbourhoods of $A$ is denoted by $\mathcal{N}(A)$.
A neighbourhood $U$ of $A$ is said to be proper if $U\neq A$.
\item
Let $U\in {\cal N}(A)$.
We denote by $A|_U$ the algebra whose underlying set is $U$ and 
whose set of term operations is given by the following clone:
\[
\Clo_n(A|_U):=\{t\rest_{U^n}\mid t\in \Clo_n(A) \text{ and }t(U^n)\subset U\}. 
\]
\item
A set $\mathcal{U}$ of neighbourhoods of $A$ is said to \emph{cover} a neighbourhood $V$ of $A$ (or $\mathcal{U}$ is a cover of $V$)
if the following condition holds:
\[
\forall r,s\in \Inv(A)\;\ (\forall U\in \mathcal{U}\ ;\ r\rest_U=s\rest_U) \Rightarrow r\rest_V=s\rest_V.
\]
Here, $r\rest_U$ denotes $r\cap U^m$ where $m$ is the arity of $r$.
\end{enumerate}
\end{df}

\begin{rem}
Let $A$ be an algebra and $U\subset A$. If $U$ is the both a subalgebra of $A$ and a neighbourhood of $A$,
then $A|_U$ is isomorphic to the subalgebra $U$ of $A$ as non-indexed algebras.
In this sense, it is not confusable to denote $A|_U$ by $U$.
\end{rem}

\begin{prop}
[cf.\ {\cite[Theorem 3.3]{Kea}}, {\cite[Theorem 3.4.6]{Beh}}]\label{covercharacterization}
Let $A$ be a finite algebra and $e_1,\ldots, e_n,e\in {\bf E}(A)$.
Then the following conditions are equivalent.
\begin{enumerate}
\item
$\{e_1(A),\ldots,e_n(A)\}$ is a cover of $e(A)$.
\item
There exist a non-negative integer $m\in \N$, term operations \mbox{$\lambda \in\Clo_m(A)$}, 
\mbox{$f_1,\ldots,f_m\in \Clo_1(A)$} and $i_1,\ldots,i_m\in \{1,\ldots,n\}$ that satisfy
\[
A\models \lambda(e_{i_1}f_{1}(x),\ldots,e_{i_m}f_{m}(x))=e(x).
\]
\end{enumerate}
\end{prop}

\begin{df}
[{\cite[Definition 5.2, 5.4]{Kea}}, {\cite[Definition 3.6.1]{Beh}}]
Let $A$ be a finite algebra and $\mathcal{U}$ be a cover of the neighbourhood $A=\id_A(A)$ of $A$.
\begin{enumerate}
\item
The algebra $A$ is said to be \emph{irreducible} if every cover of 
the algebra $A$ contains the neighbourhood $A$.
\item
A neighbourhood $U$ of $A$ is said to be \emph{irreducible} if $A|_U$ is an irreducible algebra.
\item
The cover $\mathcal{U}$ is said to be \emph{irredundant} if any $\mathcal{U}'\subsetneq \mathcal{U}$
does not cover $A$.
\item
Let $\mathcal{U}'$ be a cover of $A$. $\mathcal{U}'$ is said to be a \emph{refinement} of $\mathcal{U}$
if for any $U'\in \mathcal{U}'$ there is $U\in \mathcal{U}$ such that $U'\subset U$.
The refinement $\mathcal{U}'$ of $\mathcal{U}$ is said to be proper if 
$\mathcal{U}$ is not a refinement of $\mathcal{U}'$.
\item
The cover $\mathcal{U}$ is said to be \emph{non-refinable} if $\mathcal{U}$ has no
proper refinement. 
\end{enumerate}
\end{df}


\begin{df}[{\cite[Lemma 3.5]{Kea}}]
Let $A$ be an algebra and $U_1,\ldots, U_n\in \mathcal{N}(A)$.
An algebra $U_1\boxtimes\cdots\boxtimes U_n$ called the \emph{matrix product} of $(U_1,\ldots,U_n)$
is defined as follows:
\begin{itemize}
\item
The underlying set is the product set $\prod_{i=1}^n U_i$.
\item
The set of term operations $\Clo_m(U_1\boxtimes\cdots\boxtimes U_n)$ is the set of all operations described by
\[
\begin{pmatrix}
\begin{pmatrix}
a_{11}\\
\vdots\\
a_{n1}
\end{pmatrix}
,\ldots,
\begin{pmatrix}
a_{1m}\\
\vdots\\
a_{nm}
\end{pmatrix}
\end{pmatrix}
\longmapsto 
\begin{pmatrix}
t_1(\mathbf{a})\\
\vdots\\
t_n(\mathbf{a})
\end{pmatrix},
\]
where $t_i$ are $nm$-ary term operations of $A$ such that $t_i(A^{nm})\subset U_i$ and 
$\mathbf{a}=(a_{ij})_{\substack{1\leq i\leq n\\ 1\leq j\leq m}}$.
\end{itemize}
The case $U_1,\dots,U_n=U$, $U_1\boxtimes \dots \boxtimes U_n$ is denoted by
$U^{[n]}$.
\end{df}

\begin{thm}
[{\cite[Theorem 5.3]{Kea}}, {\cite[Theorem 3.8.1]{Beh}}, {\cite[Theorem 4.2]{Iza}}]\label{coveruniqueness}
Let $A$ be a finite algebra.
\begin{enumerate}
\item
If $\{U_1,\ldots, U_n\}$ and $\{V_1,\ldots, V_m\}$ are irredundant non-refinable cover of $A$, then $n=m$ and there exists a permutation $\sigma$
on $\{1,\dots,n\}$ that satisfies $U_i$ is isomorphic to $V_{\sigma(i)}$.
Moreover, an isomorphism $\ph_i:U_i\ra V_{\sigma(i)}$ can be chosen from
a restriction of a unary term operation of $A$.
\item
$\{U_1,\ldots, U_n\}$ and $\{V_1,\ldots, V_n\}$ be
irredundant non-refinable covers of $A$.
Then $U_1\boxtimes\cdots\boxtimes U_n$ and $V_1\boxtimes\cdots\boxtimes V_n$ are isomorphic as non-indexed algebras.
\end{enumerate}
\end{thm}

\begin{df}
Let $A$ be a finite algebra. 
\begin{enumerate}
\item
We say \emph{essential part} of $A$
the matrix product $U_1\boxtimes\cdots\boxtimes U_n$ of 
an irredundant non-refinable cover $\{U_1,\ldots,U_n\}$ of $A$.
We denote by $\Ess(A)$ the (structure of an) essential part of $A$.
\item
The algebra $A$ is said to be \emph{essential} if $\Ess(A)$ is isomorphic to $A$ as a non-indexed algebra.
\end{enumerate}
\end{df}

\begin{prop}
[{\cite[Theorem 3.3]{Kea}}]\label{}
Let $A$ be a finite algebra.
Then there exist positive integer $m$ and $A'\in {\cal N}(\Ess(A)^{[m]})$ such that $A'$ is isomorphic to $A$ as a non-indexed algebra.
\end{prop}

\begin{lem}\label{two-elements-cover}
Let $A$ be a finite algebra,
${\cal U}$ be a irredundant cover of $A$ that satisfies
$\forall U\in{\cal U}; |U|=2$. Then ${\cal U}$ is non-refinable.
\end{lem}
\begin{proof}
Assume ${\cal U}'$ is a proper refinement of ${\cal U}$.
Then ${\cal U}''=\{U\in {\cal U}'\mid |U|\neq 1\}$ is a proper subset of ${\cal U}$, since $\forall U\in{\cal U}; |U|=2$.
Because ${\cal U}'$ covers $A$, ${\cal U}''$ also covers $A$.
However, ${\cal U}''$ does not cover $A$ by irredundancy of ${\cal U}$.
It is contradiction.
\end{proof}

\begin{thm}[{\cite{Iza} Theorem 4.5}]\label{characterization-essential}
Let $A$ be a finite algebra. Then the following statements are pairwise equivalent.
\begin{enumerate}
\item
The algebra $A$ is essential.
\item
There exists a finite algebra $B$ such that $A$ is isomorphic to $\Ess(B)$.
\item
There exist $n\in\N,\lambda \in \Clo_n(A)$ and $e_1,\ldots,e_n\in {\bf E}(A)$ that satisfy
\begin{enumerate}
\item
$\{e_1(A),\dots,e_n(A)\}$ is an irredundant non-refinable cover of $A$,
\item
$A\models \lambda(e_1(x),\ldots,e_n(x))=x$,
\item
$A\models e_{i}\lambda(e_1(x_1),\ldots,e_n(x_n))=e_{i}(x_{i})$ for $i=1,\ldots,n$.
\end{enumerate}
\end{enumerate}
\end{thm}

\begin{thm}[{\cite[Theorem 4.4]{Iza}}]
\label{characterization}
Two finite algebras $A$ and $B$ are categorically equivalent if and only if 
their essential parts $\Ess(A)$ and $\Ess(B)$ are isomorphic (as non-indexed algebras).
\end{thm}

\section{Classification of matrix products of primal algebras.}
\label{s-matrix-product-primal}
In this section, we give a classification of essential algebras such that the irredundant non-refinable covers consist of primal algebras.
The classification is obtained by constructing a correspondence between such algebras and partial orders on finite sets.

First, we prepare some notations.
\begin{nt}
Let $n$ be a non-negative integer.
\begin{enumerate}
\item
We write $E_n$ the set of all clones $C$ on $\{0,1\}^n$ that satisfy the following conditions.
\begin{enumerate}
\item
For $f_1,\dots,f_n\in {\cal O}_{2}^{(m)}$,
the component-wise operation 
\[
\begin{pmatrix}
\begin{pmatrix} x_{11}\\ \vdots \\ x_{n1}\end{pmatrix}
,\ldots,
\begin{pmatrix} x_{1m}\\ \vdots \\ x_{nm}\end{pmatrix}
\end{pmatrix}
\mapsto
\begin{pmatrix}
f_1(x_{11},\dots,x_{1m})\\
\vdots \\
f_1(x_{n1},\dots,x_{nm})
\end{pmatrix}
\]
belongs to $C$.
\item
$\{\{0\}^{i-1}\times\{0,1\}\times\{0\}^{n-i}\}\mid i\in\{1,\dots,n\}\}$
is an irredundant non-refinable cover of $(\{0,1\},C)$.
\end{enumerate}
\item
We write $O_n$ the set of all partial order relations on $\{1,\dots,n\}$.
\end{enumerate}
\end{nt}

The next is the main theorem of this section.
\begin{thm}\label{matrix-product-primal}
Let $n$ be a non-negative integer.
Then the correspondence $C\mapsto \leq_C$ is
a bijection between $E_n$ and $O_n$,
where
\begin{align*}
&i_1\leq_C i_2\\
:\Longleftrightarrow&
\left[\prod_{i=1}^n r_i\in \Inv_2(C),r_{i_1}=\{0,1\}^2
\Rightarrow r_{i_2}=\{0,1\}^2\right].
\end{align*}
\end{thm}

We prove this theorem through this section.
The next lemma means the map $C\mapsto \leq_C$ is surjective.

\begin{lem}\label{order-construction}
Let $n$ be a non-negative integer.
Let $\leq$ be a partial order on $\{1,\dots,n\}$.
Let $R_{\leq}=\bigcup_{m\in\N} R_{\leq,m}$ be the following set of relations on $\{0,1\}^n$:
\[
R_{\leq,m}:=\left\{r\subset (\{0,1\}^n)^m\left| r=\prod_{i=1}^n\Delta_{{\cal E}_i}
\text{ s.t. }i_1\leq i_2\Rightarrow {\cal E}_{i_1}\supset{\cal E}_{i_2}\right.\right\}.
\]
Here, each ${\cal E}_i$ is an equivalence relation on $\{1,\dots,n\}$.
Then the following assertions hold.
\begin{enumerate}
\item
$R_{\leq}$ is a relational clone on $\{0,1\}^n$.
\item
For $i\in\{1,\dots,n\}$, $U_i:=\{0\}^{i-1}\times\{0,1\}\times\{0\}^{n-i}$
is a neighbourhood of $A:=(\{0,1\}^n,\Pol(R_\leq))$ and is a primal algebra.
\item
The family of neighbourhoods 
${\cal U}=\{U_i\mid i\in\{1,\dots,n\}\}$
is an irredundant non-refinable cover of $A$.
\item
The algebra $A$ is essential.
\end{enumerate}
\end{lem}
\begin{proof}
(1) and (2) are easily verified.

(3) $\prod_{i=1}^n \Delta_{{\cal E}_i}\neq \prod_{i=1}^n \Delta_{{\cal E}'_i}$
implies $\Delta_{{\cal E}_i}\neq\Delta_{{\cal E}'_i}$ for some $i$.
Thus ${\cal U}$ covers $A$.

Let $i_0\in\{1,\dots,n\}$.
Let 
\[
{\cal E}_i:=
\begin{cases}
\Delta & \text{if } i\geq i_0\\
\{0,1\}^2 & \text{otherwise},
\end{cases}
\]
\[
{\cal E}'_i:=
\begin{cases}
\Delta & \text{if } i> i_0\\
\{0,1\}^2 & \text{otherwise}
\end{cases}
\]
and $r:=\prod_{i=1}^n \Delta_{{\cal E}_i},s:=\prod_{i=1}^n \Delta_{{\cal E}'_i}$.
Then $r,s\in R_{\leq,2}$, $r\neq s$ and $r\rest_{U_i}=s\rest_{U_i}$ for $i\in\{1,\dots,n\}\sm\{i_0\}$ hold.
Thus the cover ${\cal U}$ is irredundant.
Non-refinability follows from Lemma \ref{two-elements-cover}.

(4) It is easily verified that the operations
$e_1,\dots,e_n$ and $d$ defined as
\[
e_i:
\begin{pmatrix} x_{1}\\ \vdots \\ x_{n}\end{pmatrix}
\mapsto 
\begin{pmatrix} \delta_{i1}x_1\\ \vdots \\ \delta_{in}x_{n}\end{pmatrix},
\]
where $\delta_{ij}x_j=0$ if $j\neq i$, $\delta_{ii}x_i=x_i$, and
\[
d:
\begin{pmatrix}
\begin{pmatrix} x_{11}\\ \vdots \\ x_{n1}\end{pmatrix}
,\ldots,
\begin{pmatrix} x_{1n}\\ \vdots \\ x_{nn}\end{pmatrix}
\end{pmatrix}
\mapsto
\begin{pmatrix}
x_{11}\\
\vdots \\
x_{nn}
\end{pmatrix}
\]
belong to $\Pol(R_\leq)$ and satisfy Equation (b) and (c) of Theorem \ref{characterization-essential} (3).
Essentiality follows from these facts and Theorem \ref{characterization-essential}.
\end{proof}

\begin{rem}
For a partial order $\leq$ on $\{1,\dots,n\}$, the clone
$\Pol(R_\leq)$ is described as follows:
Let $P_i$ denotes $\{0,1\}$, considered as the $i$-th component of $\{0,1\}^n$.
Then $\Pol_m(R_\leq)$ is the set of all operations
$(f_1,\dots,f_n)$, where $f_{i_0}:(\prod_{i\leq i_0} P_i)^m\ra P_{i_0}$,
act as
\[
(f_1,\dots,f_n):
\begin{pmatrix}
\begin{pmatrix} x_{11}\\ \vdots \\ x_{n1}\end{pmatrix}
,\ldots,
\begin{pmatrix} x_{1m}\\ \vdots \\ x_{nm}\end{pmatrix}
\end{pmatrix}
\mapsto
\begin{pmatrix}
f_1(x_{ij})_{i\leq 1,j\in\{1,\dots,m\}}\\
\vdots \\
f_n(x_{ij})_{i\leq n,j\in\{1,\dots,m\}}\\
\end{pmatrix}.
\]
\end{rem}

The clone $C=\Pol(R_\leq)$ constructed in this lemma satisfies $\leq_C=\leq$ (Notice that $\Inv_2(C)=R_{\leq,2}$ holds by Theorem \ref{pol-inv-finite}).
Thus the map defined in the statement of Theorem \ref{matrix-product-primal} is surjective.
Next, we show injectivity of the map defined in Theorem \ref{matrix-product-primal}.

\begin{lem}\label{binarypart}
Let $C\in E_n$.
Let $i_1,i_2\in\{1,\dots,n\}$. Assume there exists
$r=\prod_{i=1}^n r_i\in \Inv_m(C)$ such that
$r_{i_1}\not\supset r_{i_2}$, then $i_1\not\geq_C i_2$, namely, there exists
$r'=\prod_{i=1}^n r'_i\in\Inv_2(C)$ such that $r'_{i_1}=\Delta,r'_{i_2}=\{0,1\}^2$.
\end{lem}
\begin{proof}
Notice that $m$ must satisfy $m\geq 2$ since $C$ has all component-wise operations.
(A nullary or a unary invariant relation of a primal algebra $P$ must be $P^0$ or $P^1$ respectively.
Thus $r$ must be $\prod_{i=1}^n\{0,1\}^0$ or $\prod_{i=1}^n\{0,1\}^1$, hence $r_{i_1}\not\supset r_{i_2}$ cannot hold.)
By the assumption that $C$ has all component-wise operations, $r_i$ are described as in the form $r_i=\Delta_{{\cal E}_i}$.
${\cal E}_{i_1}\not\subset {\cal E}_{i_2}$ holds, since the assumption $r_{i_1}\not\supset r_{i_2}$.
Thus, there exists $(j_1,j_2)\in {\cal E}_{i_1}\sm{\cal E}_{i_2}$.
The projection
\begin{multline*}
r'=\pi_{\{j_1,j_2\}}(r)
:=
\left\{
\left(
\begin{pmatrix} x_{11}\\ \vdots \\ x_{n1} \end{pmatrix},
\begin{pmatrix} x_{12}\\ \vdots \\ x_{n2} \end{pmatrix}
\right)
\in (\{0,1\}^n)^2\right| \\
\left.\exists 
\begin{pmatrix}
\begin{pmatrix} y_{11}\\ \vdots \\ y_{n1} \end{pmatrix},
\dots,
\begin{pmatrix} y_{1m}\\ \vdots \\ y_{nm} \end{pmatrix}
\end{pmatrix}
\in r;
\begin{pmatrix} x_{11}\\ \vdots \\ x_{n1} \end{pmatrix}
=\begin{pmatrix} y_{1j_1}\\ \vdots \\ y_{nj_1} \end{pmatrix},
\begin{pmatrix} x_{12}\\ \vdots \\ x_{n2} \end{pmatrix}
=\begin{pmatrix} y_{1j_2}\\ \vdots \\ y_{nj_2} \end{pmatrix}
\right\}
\end{multline*}
of $r$ into components $\{j_1,j_2\}$ satisfies the condition stated in the lemma.
\end{proof}

\begin{lem}
Let $n$ be a non-negative integer, $C\in E_n$ and $i_0\in\{1,\dots,n\}$.
Then $r:=\prod_{i=1}^n r_i\in \Inv_2(C)$ holds, where
\[
r_i:=
\begin{cases}
\Delta& \text{if }i\leq_C i_0\\
\{0,1\}^2& \text{otherwise}.
\end{cases}
\]
\end{lem}
\begin{proof}
For $i\in\{1,\dots,n\}$ that satisfies $i\not\leq_C i_0$,
there exists $s_i=\prod_{i'=1}^ns_{ii'}\in\Inv(C)$ such that
$s_{ii_0}=\Delta$, $s_{ii}=\{0,1\}^2$ by the definition of $\leq_C$.
Let $\{i_1,\dots,i_k\}$ be an enumeration of $\{i\in\{1,\dots,n\}\mid i\not\leq_C i_0\}$. Then $r=s_{i_1}\circ\dots \circ s_{i_k}\in \Inv(C)$.
\end{proof}

\begin{lem}\label{construct-from-binary}
Let $n$ be a non-negative integer, $C\in E_n$.
Let $m$ be a non-negative integer and ${\cal E}_1,\dots,{\cal E}_n$ be
equivalence relations on $\{1,\dots,m\}$ that satisfy the condition
\[
i_1\leq_C i_2\ \Longrightarrow \ {\cal E}_{i_1}\supset {\cal E}_{i_2}.
\]
Then $\prod_{i=1}^n\Delta_{{\cal E}_{i}}\in\Inv_m(C)$ holds.
\end{lem}
\begin{proof}
For $i_0\in\{1,\dots,n\}$ and $j_1,j_2\in\{1,\dots,m\}$, we put
\[
\Delta^{i_0}_{j_1,j_2}:=
\left\{\left.
\begin{pmatrix}
\begin{pmatrix} x_{11}\\ \vdots \\ x_{n1}\end{pmatrix}
,\ldots,
\begin{pmatrix} x_{1m}\\ \vdots \\ x_{nm}\end{pmatrix}
\end{pmatrix}
\in (\{0,1\}^n)^m \right|
i\leq i_0\Rightarrow x_{ij_1}=x_{ij_2}
\right\}.
\]
By the previous lemma, $\Delta^{i_0}_{j_1,j_2}\in \Inv_m(C)$ holds.
Thus,
\[
\prod_{i=1}^n\Delta_{{\cal E}_i}
=\bigcap_{\substack{i\in\{1,\dots,n\}\\ (j_1,j_2)\in{\cal E}_i}} \Delta^i_{j_1,j_2}\in\Inv_m(C).
\]
\end{proof}

By these lemmas, Theorem \ref{matrix-product-primal} is proved.
\begin{cor}
The map $C\mapsto \leq_C$ defined in Theorem \ref{matrix-product-primal}
is injective.
\end{cor}
\begin{proof}
By Lemma \ref{construct-from-binary}, $\Inv(C)\supset R_{\leq_C}$ holds,
where $R_{\leq_C}$ is the relational clone defined in Lemma \ref{order-construction}.
If $\Inv(C)\supsetneq R_{\leq_C}$,
there exist $m\in\N$, $i_1,i_2\in\{1,\dots,n\}$
and equivalence relations ${\cal E}_1,\dots,{\cal E}_n$ on $\{1,\dots,m\}$ 
that satisfy $i_1\leq_C i_2$, ${\cal E}_{i_1}\not\supset{\cal E}_{i_2}$ 
and $\prod_{i=1}^n \Delta_{{\cal E}_i}\in\Inv(C)$.
By Lemma \ref{binarypart}, it implies $i_1\not\leq_C i_2$, which is contradiction.
\end{proof}
The proof of Theorem \ref{matrix-product-primal} is completed.

\section{Combinatorial description of congruence primal arithmetical algebras}
\label{s-cong-primal-arith}
In this section, we give a characterization of algebras that are categorically equivalent to an essential algebra belonging to $\bigcup_{n\in \N}E_n$, and give a concrete construction of such algebras.

\subsection{Characterization of the algebras}
In this subsection, we prove that a finite algebra is categorically equivalent to a member of $\bigcup_{n\in \N}E_n$ if and only if the algebra is congruence primal and arithmetical.

First, we show members of $E_n$ are congruence primal and arithmetical.

\begin{nt}
For $n\in\N$ and $\leq\in O_n$, we write $E_\leq:=(\{0,1\}^n,\Pol(R_\leq))$. We say $E_\leq$ the essential algebra corresponding to $\leq$.
\end{nt}

\begin{prop}\label{arith-direct}
Let $n$ be a non-negative integer and $\leq$ be a partial order on $\{1,\dots,n\}$. Then
the algebra $E_\leq$ is congruence primal and arithmetical.
\end{prop}
\begin{proof}
Put $C=\Pol(R_\leq)$.
As shown in the proof of Lemma \ref{construct-from-binary}, $\Inv(C)$ is generated by $\{r_i\mid i\in \{1,\dots,n\}\}$, where $r_i:=\prod_{i'=1}^n r_{i,i'}\in \Inv_2(C)$ and
\[
r_{i,i'}:=
\begin{cases}
\Delta& \text{if }i'\leq i\\
\{0,1\}^2& \text{otherwise}.
\end{cases}
\]
Thus $E_\leq$ is congruence primal.
All neighbourhoods belonging to an irredundant non-refinable cover of $E_\leq$ are primal, particularly are arithmetical.
An algebra having a cover consists of arithmetical algebras is arithmetical.
Thus, $E_\leq$ is arithmetical.
\end{proof}

Next, we show the converse. For a congruence primal arithmetical algebra $A$, the sketch of the construction of an algebra categorically equivalent to $A$ is as follows:
First, we consider the congruence lattice $\Con(A)$ of $A$. Next, we consider the subposet $P$ of $\Con(A)$ consists of all meet irreducible elements.
Then the matrix product of primal algebras corresponding to the reverse poset of $P$ is the algebra we want to obtain.

Let us start to describe more precise construction.
\begin{nt}
\mbox{}
\begin{itemize}
\item
For a poset $P$, we write $^rP$ the poset such that the underlying set is the same as $P$ and the order is defined as
\[
a \leq_{^rP} b\ :\Longleftrightarrow\ b\leq_P a,
\]
where $\leq_P$ is the partial order of $P$.
\item
For a distributive lattice $(L,\leq)$, $\Ir(L)$ denotes the set of all join irreducible elements of $L$.
We consider $\Ir(L)$ to be a poset equipped with the restricted partial order $\leq\rest_{\Ir(L)}=\leq\cap \Ir(L)^2$.
\item
For a finite poset $(X,\leq)$, $\J(X)$ denotes the set of all subsets $Y$ of $X$
consist of pairwise incomparable elements, i.e.,
\[
\J(X):=\{Y\subset X\mid y_1,y_2\in Y,y_1\neq y_2\Rightarrow y_1\not\leq y_2
\}.
\]
For $Y_1,Y_2\in\J(X)$, we define $Y_1\leq Y_2$ by
\[
\forall y_1\in Y_1\exists y_2\in\ Y_2; y_1\leq y_2.
\]
\end{itemize}
\end{nt}

The next correspondence between finite posets and finite distributive lattices seems to be folklore.
\begin{prop}
The following hold.
\begin{enumerate}
\item
For a finite poset $X$, $\J(X)$ is a distributive lattice.
\item
For a finite distributive lattice $L$, $\J(\Ir(L))$ is isomorphic to $L$.
\item
For a finite poset $X$, The poset $\Ir(\J(X))$ is isomorphic to $X$.
\end{enumerate}
\end{prop}

\begin{prop}\label{congruence-latice-of-essential}
Let $n$ be a non-negative integer and $\leq$ be a partial order on $\{1,\dots,n\}$.
Then the congruence lattice $\Con(E_\leq)$ is isomorphic to $\J(^r(\{1,\dots,n\},\leq))$.
\end{prop}
\begin{proof}
By the definition, $\Con(E_\leq)=R_{\leq,2}$ holds and is isomorphic to the set of all upward closed subsets of $(\{1,\dots,n\},\leq)$.
It is isomorphic to $\J(^r(\{1,\dots,n\},\leq))$.
\end{proof}

The characterization of congruence primal arithmetical algebras is stated as follows.
\begin{thm}
Let $A$ be a finite algebra. Then the following conditions are equivalent.
\begin{enumerate}
\item
An irredundant non-refinable cover of $A$ consists of primal algebras.
\item
$A$ is congruence primal and arithmetical.
\end{enumerate}
\end{thm}
\begin{proof}
(1) $\Rightarrow$ (2) follows from Proposition \ref{arith-direct} and the fact that congruence primality and arithmecity are categorically invariant properties.

(2) $\Rightarrow$ (1): Let $(\{1,\dots,n\},\leq)$ be a poset isomorphic to $^r(\Ir(\Con(A)))$. Then the essential algebra $E_\leq$ is congruence primal, arithmetical and satisfies $\Con(E_\leq)\iso\Con(A)$ by Proposition \ref{congruence-latice-of-essential}.
By Corollary 4.5 of \cite{BB}, $A$ and $E_\leq$ are categorically equivalent to each other. An irredundant non-refinable cover of $E_\leq$ consists of primal algebras,
so is an irredundant non-refinable cover of $A$.
\end{proof}

\begin{df}
For a congruence primal arithmetical algebra $A$, 
we say the (isomorphism class of) poset $^r(\Ir(\Con(A)))$ the type of $A$.
\end{df}

\subsection{Description of the algebras}

In this subsection, we exhibit a inductive description of congruence primal arithmetical algebras.

In the following paragraph,
we fix a non-negative integer $n$ and $\leq\in O_n$.
We write $C:=\Pol(R_\leq)$ and $E:=E_\leq=(\{0,1\}^n,C)$.
$P_i$ denotes the neighbourhood $\{0\}^{i-1}\times\{0,1\}\times \{0\}^{n-i}$ of $E_\leq =\{0,1\}^n$.
For notational convenience, we consider $\leq$ a partial order on $\{1,\dots,n,\infty\}$. $i_1\leq i_2$ interprets as the original partial order $\leq$ if $i_1,i_2\in\{1,\dots,n\}$ and we consider $\infty$ the top element of $\{1,\dots,n,\infty\}$.

Let $A$ be an algebra categorically equivalent to $E$.
Then there exist a positive integer $l$ and a neighbourhood $A'$ of 
$E^{[l]}\iso P_1^{[l]}\boxtimes \dots\boxtimes P_n^{[l]}$ that is isomorphic to $A$.
Thus, we can assume $A$ to be a neighbourhood of $E^{[l]}$.
Because $E^{[l_1]}$ is isomorphic to a neighbourhood of $E^{[l_2]}$ when $l_1\leq l_2$,
the case finitely many finite algebras $A_1,\dots,A_k$ categorically equivalent to $E$ is given, we may assume all of $A_i$ are neighbourhoods of $E^{[l]}$
for a common positive integer $l$. Thus, we fix a positive integer $l$ in the following context.

In this section, we use the following various notation:
\begin{itemize}
\item
For $I\subset \{1,\dots,n\}$, we write $\pi_I$ the projection $\prod_{i=1}^nP_i^{l_i}\ra \prod_{i\in I}P_i^{l_i}$.
\item
For $i\in\{1,\dots,n,\infty\}$,
$\pi_{<i}$ denotes the projection $\prod_{i'=1}^n P_{i'}^{l}\mapsto \prod_{i'<i} P_{i'}^{l}$.
Similarly, $\pi_{\leq i}$ is the projection
$\prod_{i'=1}^n P_{i'}^{l}\mapsto \prod_{i'\leq i} P_{i'}^{l}$
for $i\in\{1,\dots,n\}$.
\item
$A_{<i}:=\pi_{<i}(A)$ and $A_{\leq i}:=\pi_{\leq i}(A)$ for $A\subset \prod_{i'=1}^n P_{i'}^{l}$.
\item
For $i\in \{1,\dots,n\}$, $A\subset \prod_{i'=1}^n P_{i'}^{l}$ and $\bar{a}\in A_{<i}$, we write $A_{i,\bar{a}}$ the set $\{a_i\in P_{i}^l\mid (\bar{a},a_i)\in A_{\leq i}\}$,
where $(\bar{a},a_i)$ is the element of $\prod_{i'\leq i} P_{i'}^{l}\iso \prod_{i'< i} P_{i'}^{l}\times P_i^l$.
\end{itemize}

In the previous section, we consider the correspondence between clones on a finite set and partial orders on a finite set. In that reason, we fixed an underlying set $\{1,\dots,n\}$ of posets.
On the other hand, we consider finite posets determined by congruence lattices of finite algebras, and (downward closed) subposets of such posets in the following paragraph.
In this reason, we sometimes use modified definitions in this section: For all definitions described by the notion of posets having an underlying sets $\{1,\dots,n\}$,
we also use the same notion except replacing an abstract finite poset $P$ a concrete finite poset $(\{1,\dots,n\},\leq)$. 
For example, we write $E_P$ the essential algebra corresponding to $P$.

The next proposition gives a inductive description of neighbourhoods of $E^{[l]}$.
\begin{thm}\label{inductive-description}
Let $n$ be a non-negative integer, $\leq$ be a partial order on $\{1,\dots,n\}$ and $l$ be a positive integer.
A non-empty set $A\subset \prod_{i=1}^n P_i^{[l]}$ is a neighbourhood of $P_1^{[l]}\boxtimes\dots\boxtimes P_n^{[l]}\iso E^{[l]}$ if and only if
the following condition holds: For $\bar{x}=(x_1,\dots,x_n)\in \prod_{i=1}^n P_i^{[l]}$
\begin{equation}\label{neighbourhood-primal}
\bar{x}\in A \Longleftrightarrow
\forall i\in\{1,\dots,n\}; x_i\in A_{i,\pi_{<i}(\bar{x})}.
\end{equation}
\end{thm}
Although this theorem looks very complicated, yet describes a inductive construction (induction on the poset $(\{1,\dots,n\},\leq)$) and each algebras are described very concretely (See the remark below).

\begin{proof}
Note that
$\bar{x}\in A \Rightarrow
\forall i\in\{1,\dots,n\}; x_i\in A_{i,\pi_{<i}(\bar{x})}$
is generally true for all $A\subset E^{[l]}$.

Assume $A\in {\cal N}(E^{[l]})$, $e(E^{l]})=A$, $e=(e_1,\dots,e_n)\in {\bf E}(E^{[l]})$
and let $\bar{x}=(x_1,\dots,x_n)\in E^{[l]}$.
Note that $A_{i,\pi_{<i}(\bar{x})}=\{y_i\in P_i^l\mid e_i((x_{i'})_{i'<i},y_i)=y_i\}$.
Thus $\forall i\in\{1,\dots,n\}; x_i\in A_{i,\pi_{<i}(\bar{x})}$
implies $\bar{x}=e(\bar{x})\in A$.

Conversely, assume (\ref{neighbourhood-primal}) holds.
Then we can choose maps $f_{i,(a_{i'})_{i'<i}}:P_i^l\mapsto A_{i,(a_{i'})_{i'<i}}$ such that
$f_{i,(a_{i'})_{i'<i}}(a_i)=a_i$ for all $a_i\in A_{i,(a_{i'})_{i'<i}}$.
Here, $i$ runs $\{1,\dots,n\}$ and $(a_{i'})_{i'<i}$ runs $\prod_{i'<i}P_{i'}^l$ such that $A_{(a_{i'})_{i'<i}}\neq \emptyset$.

We define $e_i:\prod_{i'\leq i}P_{i'}^l\ra P_i^l$ as follows:
\[
e_i(x_{i'})_{i'\leq i}:=f_{i,e_{i'}(x_{i''})_{i''\leq i'}}(x_i).
\]
(This definition makes sense by induction on $i$.
Note that the case $i\in \{1,\dots,n\}$ is a minimal element with respect to $\leq$, the assumption $A\neq \emptyset$ is needed.)
Then $e=(e_1,\dots,e_n)\in {\bf E}(E^{[l]})$ and
$e(E^{[l]})=A$.
\end{proof}

\begin{rem}
A non-empty subset $A\subset \prod_{i=1}^nP_i^{l_i}$ satisfies Condition (\ref{neighbourhood-primal}) if and only if $A$ is constructed as in the following way:
We construct subsets $A_Q\subset \prod_{i\in Q}P_i^{l_i}$ for downward closed subsets $Q\subset \{1,\dots,n\}$ (with respect to $\leq$) inductively.

For the empty set: We define $A_\emptyset$ the singleton $\prod_{i\in\emptyset}P_i^{l_i}$.

The case $Q=\{j\mid j\leq i\}$: Assume $A_R\subset \prod_{j\in R}P_j^{l_j}$ are already defined for downward closed sets $R\subset Q\sm\{i\}$.
We choose non-empty subsets $A_{i,\bar{a}}\subset P_i^{l_i}$ for $\bar{a}\in A_{Q\sm\{i\}}$, and put
\[
A_Q:=\bigcup_{\bar{a}\in A_{Q\sm\{i\}}}\{\bar{a}\}\times A_{i,\bar{a}},
\]
where the first and the second entries denotes $Q\sm\{i\}$ and $i$-th components respectively.

The case $Q$ has more than two maximal elements:
we put
\[
A_Q:=\left\{\left. x\in \prod_{i\in Q}P_i^{l_i}\right|
 \pi_R(x)\in A_R\text{ for all down sets }R\subsetneq Q\right\}.
\]
Finally, we put $A:=A_{\{1,\dots,n\}}$.
\end{rem}

\begin{eg}\label{neighbourhood-1-tree}
\mbox{}
\begin{enumerate}
\item
If $\leq$ is the equality relation on $\{1,\dots,n\}$, that is $\{(i,i)\mid i\in\{1,\dots,n\}\}$,
a subset $A\subset \prod_{i=1}^n P_i^l$ is a neighbourhood of $E^l$ if and only if
$A$ can be described as $\prod_{i=1}^n A_i$ for some $A_i\subset P_i^l$.
\item \label{2tree}
Let $\leq$ be a partial order on $\{1,2,3\}$ such that $1< 2,1< 3$ and $2,3$ are incomparable.
Then a set $A\subset P_1^l\times P_2^l\times P_3^l$ satisfies Condition (\ref{neighbourhood-primal})
of Theorem \ref{inductive-description}
means that $A$ can be described as in the following form:
There exist non-empty sets $B\subset P_1^l$ and $C_b\subset P_2^l,D_b\subset P_3^l$ for $b\in B$ such that
\[
A=\bigcup_{b\in B} \{b\}\times C_b\times D_b.
\]
\item
Let $\leq$ be a partial order on $\{1,\dots,5\}$ that $1< 2< 4$, $1< 3< 4$, $3< 5$
and others are incomparable.
We choose non-empty sets $B_1\subset P_1^l$, $B_{2,b}\subset P_2^l$, $B_{3,b}\subset P_3^l$ for $b\in B_1$
and put
\[
A_{<4}:=\bigcup_{b\in B_1}\{b\}\times B_{2,b}\times B_{3,b},\ 
A_{<5}:=\bigcup_{b\in B_1}\{b\}\times B_{3,b}.
\]
Further, we choose non-empty sets $B_{4,\bar{a}}\subset P_4^l$ for $\bar{a}\in A_{<4}$
and $B_{5,\bar{a}}\subset P_5^l$ for $\bar{a}\in A_{<5}$.
Then the set
\[
A:=\left\{\left. (b_i)_{i=1}^5\in \prod_{i=1}^5 P_i^l\right|
(b_1,b_2,b_3)\in A_{<4},b_4\in B_{4,(b_1,b_2,b_3)},b_5\in B_{5,(b_1,b_3)}\right\}
\]
satisfies Condition (\ref{neighbourhood-primal}).
Conversely, subsets of $\prod_{i=1}^5 P_i^l$ satisfying Condition (\ref{neighbourhood-primal})
is constructed as above.
\end{enumerate} 
\end{eg}

Next, we consider when two neighbourhoods are isomorphic.
\begin{thm}\label{inductive-isomorphism}
Let $l$ be a positive integer and $A$ and $B$ are neighbourhoods of $E^{[l]}$.
Then $A$ and $B$ are isomorphic to each other via a term operation of $E^{[l]}$ if and only if 
there exists a family of maps $(\ph_{i,\bar{a}})_{i,\bar{a}}$ as follows:
\begin{itemize}
\item
$(i,\bar{a})$ runs tuples such that $i\in\{1,\dots,n\}$
and $\bar{a}\in A_{<i}$.
\item
For $i\in \{1,\dots,n\}$ and $\bar{a}=(a_{i'})_{i'<i}$,
$\ph_{i,\bar{a}}$ is a bijection $A_{i,\bar{a}}\ra B_{i,(\ph_{i',(a_{i''})_{i''<i'}}(a_{i'}))_{i'<i}}$.
\end{itemize}
\end{thm}

\begin{proof}
Suppose $\ph=(\ph_1,\dots,\ph_n)\in \Clo_1(E)$ and $\ph\mc\rest_A$ is an
isomorphism $A\ra B$. Define $\ph_{i,\bar{a}}:a_i\mapsto \ph_i(\bar{a},a_i)$.
Then the family $(\ph_{i,\bar{a}})_{i,\bar{a}}$ satisfies the condition.

Conversely, assume $(\ph_{i,\bar{a}})_{i,\bar{a}}$ satisfies the above condition.
We inductively define $\ph_i:\prod_{i'\leq i}P_{i'}^l\ra P_i^l$ as
\[
(x_{i'})_{i'\leq i}\mapsto \ph_{i,\ph_{i'}(x_{i''})_{i''\leq i'}}(x_i).
\]
Then $\ph=(\ph_1,\dots,\ph_n)\in \Clo_1(E)$ and $\ph\rest_A$ is an isomorphism $A\ra B$.
\end{proof}

\begin{eg}
Let $\leq$ be the partial order of Example \ref{neighbourhood-1-tree} (\ref{2tree}).
Let $B,B'\subset P_1^l$, $C_b,C'_{b'}\subset P_2^l$, $D_b,D'_{b'}\subset P_3^l$ (for $b\in B,b'\in B'$) be non-empty sets and
\[
A:=\bigcup_{b\in B}\{b\}\times C_b\times D_b,\ 
A' :=\bigcup_{b'\in B'}\{b'\}\times C'_{b'}\times D'_{b'}.
\]
Then $A$ and $A'$ are isomorphic via a term operation of $E^{[l]}$
if and only if 
\begin{itemize}
\item
There exists a bijection $\ph:B\ra B'$ that satisfies
$|C_b|=|C'_{\ph(b)}|$ and $|D_b|=|D'_{\ph(b)}|$ for all $b\in B$.
\end{itemize}
$A$ and $A'$ are isomorphic as non-indexed algebras if and only if
the above holds or
\begin{itemize}
\item
There exists a bijection $\ph:B\ra B'$ that satisfies
$|C_b|=|D'_{\ph(b)}|$ and $|D_b|=|C'_{\ph(b)}|$ for all $b\in B$.
\end{itemize}
(The second item is the case that the isomorphism is the composition
of automorphism of $E^{[l]}$ induced by the automorphism $[1\mapsto 1,2\mapsto 3,3\mapsto 2]$ of the poset $(\{1,2,3\},\leq)$
and a term operation.)
\end{eg}

At the end of this article, we consider ``minimal" algebras in a categorical
equivalence class of congruence primal arithmetical algebras.
\begin{df}
In this article, we say a finite algebra $A$ \emph{c-minimal} if there are no $U\in {\cal N}(A)\sm\{A\}$
that are categorically equivalent to $A$.
\end{df}

\begin{thm}\label{characterization-c-minimal}
A neighbourhood $A\in {\cal N}(E^{[l]})$ is categorically equivalent to $E$
if and only if for each $i\in\{1,\dots,n\}$,
there exists $\bar{a}\in A_{<i}$ such that $|A_{i,\bar{a}}|\geq 2$.
\end{thm}
\begin{proof}
Assume $A$ is categorically equivalent to $E$.
An irredundant non-refinable cover ${\cal U}$ of $A$ also be an irredundant non-refinable cover of $E^{[l]}$.
Thus, by Theorem \ref{coveruniqueness}, ${\cal U}$ is isomorphic to an irredundant non-refinable cover
$\{\{\bar{0}\}^{i-1}\times\{\bar{0},\bar{1}\}\times \{\bar{0}\}^{n-i}\mid i\in\{1,\dots,n\}\}$ of $E^{[l]}$,
where $\bar{0}=(0,\dots,0)\in \{0,1\}^l,\bar{1}=(1,\dots,1)\in \{0,1\}^l$.
Let $\{a,b\}=U_i\in {\cal N}(A)$ be a neighbourhood isomorphic to $\{\bar{0}\}^{i-1}\times\{\bar{0},\bar{1}\}\times \{\bar{0}\}^{n-i}$ via a term operation of $E^{[l]}$,
and $\bar{a}:=\pi_{<i}(a)=\pi_{<i}(b)$. Then $\pi_i(a)\neq \pi_i(b)$ and $\{\pi_i(a),\pi_i(b)\}\subset A_{i,\bar{a}}$.

Conversely, fix $i\in\{1,\dots,n\}$ and suppose $\bar{a}\in A_{<i}$, $b_i,c_i\in A_{i,\bar{a}}$, $b_i\neq c_i$.
We choose $a=(a_{i'})_{i'=1}^n\in A$ such that $\pi_{<i}(a)=\bar{a}$ and 
$b=(b_{i'})_{i'=1}^n,c=(c_{i'})_{i'=1}^n\in A$ such that  $\pi_{<i}(b)=\pi_{<i}(c)=\bar{a}$ and $b_{i'}=c_{i'}=a_{i'}$ for $i'\not\geq i$.
(Such $b$ and $c$ can be constructed by induction on $i'$.)
Define $e_i=(e_{i,i'})_{i'=1}^n\in \Clo_1(A)$ as
\[
e_{i,i'}(x):=
\begin{cases}
a_{i'}& i'\not\geq i\\
b_i   & i'=i,x_i=b_i\\
c_i   & i'=i,x_i\neq b_i\\
b_{i'}&i'>i,e_{i,i}(x)=b_i\\
c_{i'}&i'>i,e_{i,i}(x)=c_i,
\end{cases}
\] 
where $x_i$ is the $i$-th component of $x$.
\begin{claim} \label{e-eq-cover}
$\{e_1(A),\dots,e_n(A)\}$ is an irredundant non-refinable cover of $A$.
\end{claim}
Suppose
\[
r_1=\left(\prod_{i=1}^n \Delta_{{\cal E}_i}\right)\rest_{A},\ 
r_2=\left(\prod_{i=1}^n \Delta_{{\cal E}'_i}\right)\rest_{A}\in \Inv(A)
\]
and $r_1\neq r_2$.
Then there exists $i\in \{1,\dots,n\}$ such that ${\cal E}_i\neq {\cal E}'_i$, hence $e_i(r_1)\neq e_i(r_2)$.
Thus $\{e_1(A),\dots,e_n(A)\}$ covers $A$.

We consider the case
\[
{\cal E}_{i'}=
\begin{cases}
\{0,1\}^2 & (i'\not\geq i)\\
\Delta & (i'\geq i)
\end{cases}
\]
and
\[
{\cal E}'_{i'}=
\begin{cases}
\{0,1\}^2 & (i'\not> i)\\
\Delta & (i'> i).
\end{cases}
\]
Then $e_{i'}(r_1)=e_{i'}(r_2)$ for $i'\neq i$ and $e_i(r_1)\neq e_i(r_2)$ hold.
Thus the cover $\{e_1(A),\dots,e_n(A)\}$ is irredundant. Non-refinability follows from Lemma \ref{two-elements-cover}
and Claim \ref{e-eq-cover} is proved.

$\{e_1(A),\dots,e_n(A)\}$ also be a cover of $E^{[l]}$.
Therefore, $\Ess(A)\iso \Ess(E^{[l]})\iso E$.
\end{proof}
\resetclaim

\begin{cor}\label{condition-of-c-minimal}
Let $P$ be a finite poset, $A\subset \{0,1\}^P$.
Then a neighbourhood of $A$ is a c-minimal algebra of the type $P$
if and only if the following condition holds for each $i\in P$:
There exists (unique) $\bar{a}\in A_{<i}$
such that $|A_{i,\bar{a}}|=2$, and $|A_{i,\bar{x}}|=1$ for $\bar{x}\in A_{<i}\sm\{\bar{a}\}$.
\end{cor}

The next c-minimal algebra is essentially constructed in \cite{BB}.
\begin{eg}[{\cite{BB} page 187}]\label{congruence-is-c-minimal}
Let $P$ be a finite poset.
Then the set
\[
\M(P):=\{\alpha\in \{0,1\}^P\mid \alpha(i)=1,\alpha(j)=1,i\neq j\Rightarrow i\not\leq j,j\not\leq i\}
\]
is a neighbourhood of $E_P$.
Moreover, it is a c-minimal algebra of the type $P$.
\end{eg}

\begin{eg}
\mbox{}
\begin{enumerate}
\item
Let $\leq$ be a partial order on $\{1,2,3\}$ such that $1<2<3$.
Then there are two (up to isomorphism) c-minimal algebras of the type $(\{1,2,3\},\leq)$.
That are the following neighbourhoods of $E=\boxtimes_{i=1}^3 P_i$.
\[
\{(0,0,0),(1,0,0),(0,1,0),(0,0,1)\},\ \{(0,0,0),(1,0,0),(0,1,0),(0,1,1)\}.
\]
\item
Let $\leq$ be a partial order on $\{1,2,3\}$ such that $1<2,1<3$ and $2$ and $3$ are incomparable.
Then there are two (up to isomorphism) c-minimal algebras of the type $(\{1,2,3\},\leq)$.
That are the following neighbourhoods of $E=\boxtimes_{i=1}^3 P_i$.
\[
\{(0,0,0),(1,0,0),(0,1,0),(0,0,1),(0,1,1)\},\ \{(0,0,0),(1,0,0),(0,1,0),(1,0,1)\}.
\]
\item \label{1-cotree}
Let $\leq$ be a partial order on $\{1,\dots,n\}$ such that $n$ is a top element and others are pairwise incomparable.
Then there is unique (up to isomorphism) c-minimal algebras of the type $(\{1,\dots,n\},\leq)$.
That is a neighbourhood $\{0,1\}^{n-1}\times\{0\}\cup \{0\}^{n-1}\times\{1\}$ of $E=\boxtimes_{i=1}^n P_i$.

\item
Let $\leq$ be a partial order on $\{1,2,3,4,5\}$ such that $1$ is a bottom, $5$ is a top and $2,3,4$ are
pairwise incomparable, namely, $\{1,2,3,4,5\}$ is isomorphic to $M_3$.
Then there is a $7$-element c-minimal algebra of the type $M_3$,
which is
\[
\{0\}\times \{0,1\}\times\{0,1\}\times \{0\} \times\{0\}\cup
\{0\}\times \{0\}\times \{0\}\times\{0,1\}\times \{0\}\cup
\{(0,0,0,0,1)\}.
\]
However, a c-minimal algebra of the type $(\{2,3,4,5\},\leq\cap\{2,3,4,5\}^2)$
(is unique up to isomorphism and) has $9$($>7$) elements ((\ref{1-cotree}) above.).
\end{enumerate}
\end{eg}

By the description theorems
(Theorem \ref{inductive-description}, \ref{inductive-isomorphism}, and Corollary \ref{condition-of-c-minimal}),
we can assert many properties of c-minimal congruence primal arithmetical algebras. 
At the end of this paper, we describe several extremal cases.

First, we state a technical lemma.
\begin{lem}
Let $P$ be a finite poset, $P'\subset P$ be a downward closed subset.
Let $A'\subset E_{P'}=\{0,1\}^{P'}$ be a c-minimal algebra of the type $P'$.
Then there exists a c-minimal algebra $A\subset E_P=\{0,1\}^P$ of the type $P$
such that $\pi_{P'}(A)=A'$.
\end{lem}
\begin{proof}
It is enough to prove for the case $|P\sm P'|=1$.
Let $\{i\}=P\sm P'$ and $a\in \pi_{<i}(A')$.
Then
\[
A:=A'\times\{0\}\cup \{x\in \{0,1\}^P\mid \pi_{<i}(x)=a\}
\]
is a c-minimal algebra of the type $P$ by Corollary \ref{condition-of-c-minimal}.
\end{proof}

\begin{thm}
Let $P$ be a finite poset. Then the following conditions are pairwise equivalent.
\begin{enumerate}
\item
All c-minimal algebras of the type $P$ have cardinality $|P|+1$.
\item
The poset $P$ is totally ordered.
\item
The distributive lattice $\J(P)$ is totally ordered.
\end{enumerate}
\end{thm}
\begin{proof}
(2) $\Rightarrow$ (1):
We prove by induction on $|P|$.
The case $|P|=0$, c-minimal algebras of the type $P$ are
unique and is the one-element algebra with nullary term operation.

Assume (2) $\Rightarrow$ (1) holds for the case $|P|=n$.
Let $P$ be a totally ordered set with $|P|=n+1$.
We may assume without loss of generality that $P=\{0,\dots,n\}$
with restriction of the canonical order on the set of integers.
By Corollary \ref{condition-of-c-minimal}, each (underlying set of)
c-minimal algebra of the type $P$ is described as
\[
A'\times\{0\}\cup \{a\}\times\{1\}
\]
for some c-minimal algebra $A'$ of the type $P'=\{0,\dots,n-1\}$ and $a\in A'$.
Here, the first and the second entries denote $P'$ and $n$-th components respectively.
By the induction hypothesis, a c-minimal algebra $A'$ of the type
$P'$ has the cardinality $n+1$. Therefore
$|A|=|A'|+1=n+2$.

(1) $\Rightarrow$ (2):
Assume $P$ is not totally ordered.
Let $(i,j)\in P^2$ be a minimal incomparable pair,
that is a pair of elements of $P$ that satisfies
\[
\forall i'\leq i,\forall j'\leq j\ [
(i',j')\text{ are incomparable }\Longleftrightarrow\ (i',j')=(i,j)].
\]
Note that $\{x\in P\mid x<i\}=\{x\in P\mid x<j\}$ holds.
(If $x<i$, then $(x,j)$ is comparable since minimality of $(i,j)$.
If $j\leq x$, then $j\leq x<i$. It contradicts incomparability of $(i,j)$.
Thus $x<j$.)

Put $P'=\{x\in P\mid x<i\}$.
Let $A'\subset E_{P'}$ be a c-minimal algebra of the type $P'$.
Then 
\[
\tilde{A}=A'\times \{0\}\times\{0\} \cup \{a\}\times \{0,1\}\times\{0,1\}
\]
is a c-minimal algebra of the type $\tilde{P}:=P'\cup\{i,j\}$.
Here, the first, the second and the third entries denote $P',i$ and $j$
components respectively.
In this setting,
\[
|\tilde{A}|=|A'|+3> |P'\cup\{i,j\}|+1.
\]
Let $A$ be a c-minimal algebra of the type $P$ such that $\pi_{\tilde{P}}(A)=\tilde{A}$.
Then $|A|\geq |\tilde{A}|+|P\sm\tilde{P}|>|P|+1$.

(2) $\Leftrightarrow$ (3) is well known.
\end{proof}


\begin{thm}
Let $P$ be a finite poset. Then the following conditions are pairwise equivalent.
\begin{enumerate}
\item
The algebra $E_P$ is c-minimal.
\item
The poset $P$ is discrete, i.e., distinct elements are incomparable.
\item
The distributive lattice $\J(P)$ is (isomorphic to the underlying lattice of) a Boolean algebra.
\item
The algebra $E_P$ is non-indexed product of primal algebras.
\end{enumerate}
\end{thm}
\begin{proof}
(1) $\Rightarrow$ (2): Let $i,j\in P$ and $i<j$.
Then $\{(a_{k})_{k\in P}\in E_P\mid (a_i,a_j)\neq (1,1)\}$
is a type $P$ proper neighbourhood of $E_P$.

(2) $\Rightarrow$ (1):
An algebra $A$ categorically equivalent to $E_P$ is of the form
$A=\prod_{i\in P}A_i$ ($|A_i|\geq 2$ for all $i\in P$) with component-wise operations.
Thus the case $|A_i|=2$ for all $i\in P$, namely, $A=E_P$ is c-minimal.

(2) $\Leftrightarrow$ (3) is well known.

(2) $\Leftrightarrow$ (4) directly follows from the definition of $E_P$.
\end{proof}

\begin{thm}
Let $P$ be a finite poset.
Then the isomorphism classes of c-minimal algebras of the type $P$ are unique
if and only if
$P$ is a depth 1 co-forest, i.e., $|\{x\in P\mid x>a\}|\leq 1$ for all $a\in P$.
\end{thm}
\begin{proof}
We write $P_{<l}:=\{x\in P\mid x<l\}$ and $P_{\leq l}:=\{x\in P\mid x\leq l\}$ for $l\in P$.

Depth $1$ co-forest $\Rightarrow$ Uniqueness:
We prove by induction on $|P|$.
The case $P=\emptyset$ is clear.

Let $P$ be a poset and $i\in P$ be a maximal element of $P$.
Then $P'=P\sm\{i\}$ is a downward closed set and a depth 1 co-forest.
Put $Q=\{x\in P'\mid x\not< i\}\neq\emptyset$.

By the induction hypothesis, c-minimal algebras of the type $P'$
are unique up to isomorphism. Let $A'\in{\cal N}(E_{P'})$ be a c-minimal algebra of the type $P'$.

%
%

By the assumption that $P$ is a depth 1 co-forest, each $j<i$ is incomparable to any elements of $P'=P\sm\{i\}$.
Thus, $A'$ can be written as
\[
A'=\pi_Q(A')\times \{0,1\}^{P_{<i}}.
\]
Therefore, a c-minimal neighbourhood of $E_P$ of the type $P$ is described as in the form
\[
A_{a_0,\eps}=\bigcup_{a\in \{0,1\}^{P_{<i}}}(\pi_Q(A')\times\{a\}\times \{\eps_a\})
 \cup (\pi_Q(A')\times\{a_0\}\times \{0,1\}),
\]
where $\eps_a\in\{0,1\}$ for $a\in\{0,1\}^{P_{<i}}$, and $a_0\in \pi_{<i}(A')=\{0,1\}^{P_{<i}}$.
Here, the first, the second and the third entries denotes $Q$, $P_{<i}$ and $i$-th components.
By this description, it is easily verified that the structure of $A_{a_0,\eps}$ does not depend on $a_0$ or $\eps_a$.
Thus, the structure of c-minimal algebras of the type $P$ is unique.

Uniqueness $\Rightarrow$ Depth 1 co-forest:
Let $P$ be a poset that is not a depth 1 co-forest.
We show that there is a c-minimal algebra that is not isomorphic to $\M(P)$ in Example \ref{congruence-is-c-minimal}.
First, notice that $\M(P)$ satisfies the following property:

(*) For $i,j,k\in P$ and $a,b\in A_{\leq i}$ such that $i<j,i<k,j\neq k$ and $\pi_{<i}(a)=\pi_{<i}(b), a\neq b$,
one of the sets
\[
\{(x_1,x_2)\in\{0,1\}^2\mid \exists(y_l)_{l\in P}; (y_j,y_k)=(x_1,x_2),(y_l)_{l\leq i}=a\}
\]
or
\[
\{(x_1,x_2)\in\{0,1\}^2\mid \exists(y_l)_{l\in P}; (y_j,y_k)=(x_1,x_2),(y_l)_{l\leq i}=b\}
\]
is a singleton.

Thus it is enough to show that there is a c-minimal algebra of the type $P$ that does not satisfy (*).

Let $i,j,k\in P$ such that $i<j,i<k,j\neq k$, $A_1\subset \{0,1\}^{P_{\leq i}}$ be a c-minimal algebra of the type $P_{\leq i}$ and $a_{<i}\in \{0,1\}^{P_{<i}}$ be the element that satisfies
$(a_{<i},0),(a_{<i},1)\in A_1$.

\textbf{Case 1}. $j$ and $k$ are comparable.

We may assume $j<k$.
Let $A_2'$ be a c-minimal algebra of the type $P_{<j}$ such that $\pi_{\leq i}(A_2')=A_1$.
Let $b\in A_2'$ be an element such that $\pi_{\leq i}(b)=(a_{<0},0)$ and define
\[
A_2:=A_2'\times\{0\}\cup \{x\in \{0,1\}^{P_{<j}}\mid \pi_{<j}(x)=\pi_{<j}(b)\}\times\{1\}.
\]
Here, the first and the second entries denotes $P_{<j}$ and $j$-th components.
Then $A_2\subset \{0,1\}^{P_{\leq j}}$ is a c-minimal algebra of the type $P_{\leq j}$.

Next, let $A_3'$ be a c-minimal algebra of the type $P_{<k}$ such that $\pi_{\leq j}(A_3')=A_2$, $c\in A_3'$ such that $\pi_{\leq i}(c)=(a_{<i},1)$
and define
\[
A_3:=A_3'\times \{0\}\cup \{x\in\{0,1\}^{P_{<k}}\mid \pi_{<k}(x)=\pi_{<k}(c)\}\times\{1\}.
\]
Here, the first and the second entries denotes $P_{<k}$ and $k$-th components.
Then $A_3$ is a c-minimal algebra of the type $P_{\leq k}$.
Finally, let $A$ be a c-minimal algebra of the type $P$ that $\pi_{\leq k}(A)=A_3$. Then $A$ does not satisfy Condition (*).

\textbf{Case 2}. $j$ and $k$ are incomparable.

Put $Q_2':=P_{<j}\cup P_{<k}$ and $Q_2:=Q_2'\cup\{j,k\}$.
Let $A_2'$ be a c-minimal algebra of the type $Q_2'$ such that $\pi_{<i}(A_2')=A_1$.
Let $b\in \pi_{<j}(A_2')$ and $c\in\pi_{<k}(A_2')$ be elements that satisfy
$\pi_{<i}(b)=(a_{<i},0),\pi_{<i}(c)=(a_{<i},1)$.
We define
\[
A_2:=\{(x_l)_{l\in Q_2}\in \{0,1\}^{Q_2}\mid
(x_j,x_k)=(0,0)\text{ or }(x_l)_{l<j}=b \text{ or }(x_l)_{l<k}=c \}.
\]
Then $A_2$ is a c-minimal algebra of the type $Q_2$, since $A_2\subset \{0,1\}^{Q_2}=E_{Q_2}$ satisfies the condition stated in Corollary \ref{condition-of-c-minimal}.
Let $A$ be a c-minimal algebra of the type $P$ such that $\pi_{P_2}(A)=A_2$. Then $A$ does not satisfy Condition (*).
\end{proof}

\end{document}